\documentclass{amsart} 
\usepackage{eulervm}
\usepackage{times}
\usepackage[latin9]{inputenc}
\usepackage{mathrsfs}
\usepackage{amstext}
\usepackage[pdftex]{graphicx}
\usepackage[all]{xy}
\usepackage[mathscr]{eucal}
\usepackage{amsmath,latexsym,oldlfont}
\usepackage{amsthm}
\usepackage{amssymb}
\usepackage{tikz}
\usetikzlibrary{calc}
\usepackage{mathdots}
\usepackage[colorlinks,linkcolor=blue,citecolor=red,linktocpage=true]{hyperref} 
\numberwithin{equation}{section}
\numberwithin{figure}{section}
\theoremstyle{plain}
\newtheorem{thm}{\protect\theoremname}[section]
 \theoremstyle{remark}
  \newtheorem{rem}[thm]{\protect\remarkname}
  \theoremstyle{plain}
  \newtheorem{cor}[thm]{\protect\corollaryname}
  \theoremstyle{plain}
  \newtheorem{lem}[thm]{\protect\lemmaname}
  \theoremstyle{plain}
  \newtheorem{prop}[thm]{\protect\propositionname}
    \theoremstyle{definition}
  \newtheorem{dfn}[thm]{\protect\definitionname}
   \theoremstyle{remark}
  
  \theoremstyle{plain}
  \newtheorem{expl}[thm]{\protect\examplename}
  \providecommand{\corollaryname}{Corollary}
  \providecommand{\examplename}{Example}
  \providecommand{\lemmaname}{Lemma}
  \providecommand{\propositionname}{Proposition}
  \providecommand{\remarkname}{Remark}
\providecommand{\theoremname}{Theorem}
\providecommand{\Algorithmname}{Algorithm}
  \providecommand{\definitionname}{Definition}
   
\DeclareMathOperator{\GL}{GL}

\DeclareMathOperator{\Spec}{Spec}

\DeclareMathOperator{\sgn}{sgn}

\DeclareMathOperator{\hes}{Hess}
\DeclareMathOperator{\dg}{diag}
\DeclareMathOperator{\Ad}{Ad}

\def\fg{{\mathfrak g}}
\def\fgl{{\mathfrak gl}}

\def\ffb{{\mathfrak b}}
\def\fft{{\mathfrak t}}

\def\k{{\mathbb C}}

\title{On the T-equivariant cohomology of Hessenberg varieties}
\author{Daniel S\'anchez Arg\'aez}
\address{
Departamento de Matem\'aticas,
Universidad Aut\'onoma Metropolitana-I,
09340 Ciudad de M\'exico,  M\'exico}
\email{dasa@xanum.uam.mx}
\author{Felipe Zald\'{\i}var}
\address{
Departamento de Matem\'aticas,
Universidad Aut\'onoma Metropolitana-I,
09340 Ciudad de M\'exico,  M\'exico}
\email{fz@xanum.uam.mx}
\date{}
  

\begin{document}

\begin{abstract}
For an endomorphism $s:V\rightarrow V$ of a finite dimensional complex vector space and an action of a torus $T$ on the full flag variety $\GL_n(\k)/B$, we give a description of its fixed point set when $s$ is semisimple or regular nilpotent. We also compute the one dimensional orbits of this action on the Hessenberg subvariety $\hes(s,h)\subseteq \GL_n(\k)/B$ for any Hessenberg function $h$. For the action of the one dimensional torus $S$ and a  regular nilpotent endomorphism $N:V\rightarrow V$, we give a new computation of  the equivariant cohomology of the Hessenberg variety  $\hes(N,h)$ for any Hessenberg function using determinantal conditions. 
\end{abstract}
\maketitle

\section{Introduction}
Let $V$ be a $\k$- vector space of finite dimension $\dim V=n$, and let $\GL_n(\k)/B$ denote the full flag variety, where $B$ is the subgroup of upper-triangular matrices. Let $\fg$, $\ffb$ denote the Lie algebras corresponding to $\GL_n(\k)$ and $B$ respectively. For $X\in\fg$ and $H\subset \fg$ a $\ffb$-submodule such that $\ffb\subset H$, consider the \textit{Hessenberg variety}
$$\hes(X,H)=\lbrace gB\in \GL_n(\k)/B: \Ad(g^{-1})\cdot X\in H\rbrace,$$
where $\Ad: \GL_n(\k) \rightarrow \GL_n (\fg)$ is the adjoint representation. There is another description $\hes(X,h)$ of a Hessenberg variety using a \textit{Hessenberg function} $h:\{1,\ldots,n\}\rightarrow \{1,\ldots,n\}$ (see Section \ref{sec2}) with $\hes(X,h)=\hes(X,H)$ for $H$ a subspace that depends on $h$.  We consider the $n$-dimensional torus $T\subseteq B$ given by the diagonal matrices and its action on $\GL_n(\k)/B$ by left multiplication. Let $\fft$ denote its Lie algebra. By Tymoczko \cite[Proposition 2.8]{Tym2} there is a description of the fixed point set $(\GL_n(\k)/B)^T$ and the one dimensional orbits associated to the action of $T$. Goresky-MacPherson in \cite{GM}  give a description $H_{S}^*(Y;\k)$ for $Y$ a $GM$-space and a one-dimensional torus $S\simeq \k^*$. Since $\GL_n(\k)/B$ is a $GM$-space,  in \cite{phd} Insko  computes $H_{S}^*(\GL_n(\k)/B;\k)\simeq \k[x_1,\ldots,x_n,t]/I_{1}$ where $I_1$ is an ideal that gives a combinatorial characterization of the set $(\GL_n(\k)/B)^S$. 
Assuming that $N$ is nilpotent and that $H_{S}^*(\hes(N,h);\k)$ is generated by its  cohomology classes of degree two, then $\hes(N,h)$ is a $GM$-subspace of $\GL_n(\k)/B$ and therefore $$H_{S}^*(\hes(N,h);\k)\simeq H_{S}^*(\GL_n(\k)/B;\k)/I$$ 
for some ideal $I$. In \cite[Theorem 5.11]{phd} Insko defines  an ideal $I_h\in \k[x_1,\ldots,x_n,t]$ using \textit{Hessenberg diagrams} to describe the ideal $I$ as $I=I_1+I_h$ and computes $H_{S}^*(\hes(N,h);\k)$.

Since $\hes(X,h)$ is a closed subvariety of $\GL_n(\k)/B$, our contribution starts by
considering which of the above properties for the full flag variety are inherited to $\hes(X,h)$. If $X$ is a semisimple operator $X=\dg(\lambda_1, \ldots,\lambda_n)$ the action of $T$ restricts to an action on $\hes(X,h)$ by Theorem \ref{teo3.1} and $(\hes(X,h))^T\subseteq (\GL_n(\k)/B)^T$.  Theorem \ref{teo2}  identifies $(\hes(X,h))^T=(\GL_n(\k)/B)^T$ and the one dimensional orbits in $\hes(X,h)$. On the other hand, if $N$ is a regular nilpotent operator, there are some similar results as in the semisimple case although the action of $T$ does not restrict to an action on $\hes(N,h)$.  In Theorem \ref{teo4} we compute which fixed points of $\GL_n(\k)/B$ are contained in $\hes(N,h)$. 
Now, if instead of an $n$-torus we take a one dimensional torus $S=\lbrace \dg(t, t^2, \ldots, t^n):t\in \k^*\rbrace$, then $\hes(N,h)$ is $S$-invariant under the $S$ action on $\GL_n(\k)/B$ and we identify $(\hes(N,h))^S$ in Corollary \ref{cor4.5}. Also, Theorem \ref{teo4} allow us to give another description of the fixed points of $\hes(N,h)$ using an ideal $\tilde{I}_h\subset \k[x_1,\ldots,x_n]$. We define $\tilde{I}_h$  by monomials whose zeros contain the values of a permutation $w$ in $S_n$. Indeed, if $\overline{S_n}=\lbrace \overline{w}=(w(1),\ldots, w(n)): w\in S_n\rbrace$, then $\overline{w}\in V(\tilde{I}_n)\cap\overline{S}_n$ if and only if $wB\in \hes(N,h)^S$, where $V(\tilde{I}_n)$ is the (projective) zero-set of the ideal $\tilde{I}_n$. This description of $\hes(N,h)^S$ via $\tilde{I}_h$ allow us to identify $V(\tilde{I}_h)\cap\overline{S}_n$ with $V(I_1+I_h)$
under  a suitable evaluation and therefore to compute its $S$-equivariant cohomology  $H_{S}^*(\hes(N,h);\k)$ using $\tilde{I}_h$. 

The paper is organized as follows: 
Section \ref{sec2} recalls two basic definitions of the Hessenberg variety and contains a description of the defining ideal of $\hes(X,H)$ due to Insko \cite[Theorem 10]{Insko}.  We give a simpler proof of this theorem using determinantal conditions. In addition, in this section we reinterpret some results of GKM-theory applied to  an $n$-torus $T\subset B$ acting on the full flag variety and we compute the resulting moment graph.  

For the natural action of a full dimensional algebraic torus $T\subseteq B$ on the flag variety $\GL_n(\k)/B$, in Section \ref{sec3} we start by recalling how this action restricts to an action of $T$ on the Hessenberg closed subvariety of $\GL_n(\k)/B$  
in Proposition \ref{prop3.1}.  Theorems \ref{teo3.1} and \ref{teo2} describe the one-dimensional orbits of the action of $T$ on $\hes(X,H)$. 

Section \ref{sec4} considers a regular nilpotent operator $N$, with corresponding Hessenberg variety $\hes(N,h)$ for $h$ any Hessenberg function. In this section we continue considering the action of $T$ on $\GL_n(\k)/B$. We analyze  when a fixed point of $\GL_n(\k)/B$ under $T$ belongs to $\hes(N,h)$ although the action of $T$ does not restrict to $\hes(N,h)$. Theorem \ref{teo4} characterizes the flags $[w]=wB\in (\GL_n(\k)/B)^T$  such that $[w]\in\hes(N,h)$ for $w\in S_n$.

Finally,  in Section \ref{sec5} we consider the  action of $S=\lbrace \dg(t, t^2, \ldots, t^n):t\in \k^*\rbrace$ on $\GL_n(\k)/B$. In this case, $\hes(N,h)$ is an $S$-invariant subvariety and Corollary \ref{cor4.5} describes the fixed point set $\hes(N,h)$. This section includes a Corollary \ref{cor4.6}, a consequence of \cite[Theorem 4.14]{phd} that computes the $S$-equivariant class of $\hes(N,h)$. Also, once the ideal $\tilde{I}_h$ is defined using Corollary \ref{cor4.5}  we give a new proof of \cite[Theorem 5.11]{phd}.

\section{Preliminaries}\label{sec2}
 
\begin{dfn} 
Let $V$ be a complex vector space  of finite dimension $\dim V=n$. 
A \textit{full flag} $V_{\bullet}:= V_1 \subsetneq V_2 \subsetneq \ldots \subsetneq V_{n-1}\subsetneq V$  is a sequence of  nested subspaces such that $\dim V_i=i$. The collection of all full flags in $V$ 
 is a projective variety,  indeed a determinantal variety. Choosing a basis of $V\simeq{\mathbb C}^n$, 
 there is an isomorphism between the full flag variety and the homogeneous space $\GL_n({\mathbb C})/B$, where $B$ is the subgroup of upper-triangular matrices.
\end{dfn}

For $V={\mathbb C}^n$ and $g\in \GL_n({\mathbb C})$, let $V_{\bullet}(g)$ be the full flag where each subspace is the span $V_j (g)$ of the first $j$ column vectors of $g$. 
Let $S_n$ be the group of permutations of $n$ letters and fix the standard basis $\left\lbrace e_1,\cdots, e_n\right\rbrace$ of $V$. For $w$ a permutation in $S_n$ viewed as a permutation matrix $w\in \GL_n({\mathbb C})$ we let $[w]\in \GL_n({\mathbb C})/B$ denote the flag given by $we_i=e_{w(i)}$. For the transposition  $s_{j,k}\in S_n$  that interchanges $j$ and $k$ and 
for $0\neq c\in {\mathbb C}$, let $G_{j,k}(c)\in \GL_n({\mathbb C})$ be the matrix given by 
\[
{(G_{jk}(c))}_{il}=
\begin{cases}
1 & \text{ if } i=l, \\
c & \text{ if } i=j \text{ and } k=l,\\ 
0 & \text{ otherwise.}
\end{cases}
\]
For an $n$-torus $T\subset B$ acting on $\GL_n(\k)/B$  by left multiplication we are interested on the $T$-equivariant cohomology of the flag variety $\GL_n(\k)/B$. GKM theory gives techniques for computing the $T$-equivariant cohomology, and one useful result is given by Tymoczko \cite{Tym2}: Let $T\subset B \subset \GL_n(\k)$ an $n$-torus, $T=\left\lbrace \dg(T_1,\cdots, T_n): T_i\in \k^{*}\right\rbrace$ and $\fg$, $\ffb$ and $\fft$ be Lie algebras corresponding to $\GL_n(\k)$, $B$, $T$ respectively, where $\fft=\left\lbrace (t_1,\cdots, t_n):t_i\in \k\right\rbrace$.
In \cite[Proposition 2.8]{Tym2},  Tymoczko  proved the next proposition:
\begin{prop}[Tymoczko {\cite[Proposition 2.8]{Tym2}}]\label{prop1} 
{\rm (1)} In $\GL_n(\k)/B$, the vertices of the moment graph  are the points $[w]\in \GL_n(\k)/B$ for $w\in S_n$.
\smallskip

\noindent{\rm (2)} 
If $j<k$, there is an edge of the moment graph  from $[w]$ to $[s_{j,k}w]$ if and only if $w^{-1}(j)>w^{-1}(k)$. The points of this edge correspond to the flags $[G_{j,k}(c)w]\in \GL_n(\k)/B$ with $c\neq 0$. This edge is labeled $t_j -t_k$.
\end{prop}

\begin{dfn}
A \textit{Hessenberg function} $h$ is a function $h:\left\lbrace 1,2,\cdots, n\right\rbrace \rightarrow \left\lbrace 1,2,\cdots, n\right\rbrace $ such that $j\leq h(j)$ for all $1\le j \leq n$ and $h(j)\leq h(j+1)$ for all $1\leq j\leq n-1$. For a fixed element $X$ in the Lie algebra $\fgl_n$ of $\GL_n(\k)$, the corresponding \textit{Hessenberg variety} is the flag variety
\begin{equation}\label{eq2.1}
\hes(X,h)=\left\lbrace gB\in G/B: XV_{j}(g)\subseteq V_{h(j)}(g)\right\rbrace.
\end{equation}
\end{dfn}
C. Procesi \cite{Proc} has a more general definition of a Hessenberg variety:
Let $G$ be a semisimple reductive algebraic group over $\k$, $B\subseteq G$ a Borel subgroup,  $\fg$  the Lie algebra of $G$ and   $\ffb\subseteq \fg$ the Lie subalgebra corresponding to $B$. A \textit{Hessenberg space} is a $\ffb$-submodule  $H\subseteq\fg$ such that $\ffb\subseteq H$. Given a Hessenberg  space $H$ and a fixed element $X\in \fg$, the corresponding \textit{Hessenberg variety} is the subvariety of $G/B$ defined by
\begin{equation}\label{eq2.2.0}
  \hes (X,H)=\left\lbrace gB\in G/B : \Ad (g^{-1})\cdot X\in H\right \rbrace,
\end{equation}
where $\Ad:G\rightarrow\GL(\fg)$ is the adjoint representation of $G$.
When  $G=\GL_n(\k)$  definition \eqref{eq2.2.0} is equivalent to \eqref{eq2.1}  noticing that for $\fg=\fgl_n$ and $\ffb=\ffb_n$ the Borel subalgebra of upper triangular matrices,  a 
Hessenberg function $h$ defines the Hessenberg space
$$H_h=\ffb \oplus \bigoplus_{j\leq i\leq h(j)} \fg_{e_i-e_j}=\big\{(x_{ij})\in\fgl_n:x_{ij}=0\;\text{for all $i>h(j)$}\big\},$$
and for an element $X\in\fgl_n$, using that for $G=\GL_n(\k)$ the adjoint representation is given by conjugation there is an equality $\hes(X,h)=\hes(X, H_h)$ as sets and as varieties as shown by
 E. Insko, J. Tymoczko and A. Woo \cite[Proposition 8]{Insko}. In \cite[Theorem 10]{Insko} there is a description of the defining ideal of the Hessenberg variety as
\begin{equation}\label{eq2.2}
I_{X, H_h}=\left\langle d_{(1,\ldots, n)(u_1,\ldots,u_{i-1}, Xu_j, u_{i+1},\ldots, u_n)}: i>h(j)\right\rangle
\end{equation}
where $ u_1, \cdots, u_n$ are the column vectors of  
a generic $n\times n$ matrix $Z=(z_{ij})$ 
and $d_{(1,\ldots, n)(u_1,\ldots,u_{i-1}, Xu_j, u_{i+1},\ldots, u_n)}$ is the determinant of the matrix with columns 
$u_1,\ldots,u_{i-1}$, $Xu_j, u_{i+1},\ldots, u_n$, namely $d\in R=\k[z_{ij}]$. Since we are considering flags, then $d\neq 0$ and $I_{X,H_h}\subset \k[z_{ij},d^{-1}]$. We start by giving a different description of $I_{X,H_h}$ and to do this we start by giving a description of $\hes(X,h)$. Given $gB\in G/B$, let $v_1,\ldots,v_n$ be the columns of $g$ and
 $V_{\bullet}(g)$  the flag defined by $gB$; observe that 
\begin{align}\label{eq2.3}
\begin{split}
V_{\bullet}(g)\in \hes(X,h) &\Leftrightarrow X\cdot V_j(g) \subseteq V_{h(j)}(g) \text{ for all $1\leq j < n$} \\
& \Leftrightarrow Xv_1, \ldots, Xv_j \in V_{h(j)}(g)\\
& \Leftrightarrow \left\langle v_1, \ldots v_{h(j)}, Xv_1,\ldots, Xv_j\right\rangle\! =\! V_{h(j)}  \text{ for all $1\leq j < n$}.
\end{split}
\end{align}
This is equivalent to say that every size $h(j)+1$ minor of the matrix whose columns are the vectors $\left\lbrace v_1,\cdots, v_{h(j)}, Xv_1, \cdots, Xv_j\right\rbrace$ vanish. Now, for $l<j$ then $h(l)\leq h(j)$, so $V_{h(l)}(g)\subseteq V_{h(j)}(g)$. Then, for $V_{\bullet}(g)\in \hes(X,h)$ and for each $l$ and $l+1$ we have that $X V_l(g)\subseteq V _{h(l)}(g)\subseteq V_{h(l+1)}(g)$. Therefore, the inclusion 
$XV_{l+1}(g) \subseteq V _{h(l+1)}(g)$  is determined by the condition  $Xv_{l+1}\in V_{h(l+1)}(g)$. Hence,
\begin{align}\label{eq2.4}
\begin{split}
X\cdot V_j(g) \subseteq V_{h(j)}(g)  &\Leftrightarrow \left\langle v_1, \ldots v_{h(j)}, Xv_1,\ldots, Xv_j\right\rangle =V_{h(j)}(g)\\
&\Leftrightarrow Xv_j\in V_{h(j)}(g)\text{ for all $1\leq j < n$}.
\end{split}
\end{align}
The last condition implies that every size $h(j)+1$ minor of the matrix formed by  the columns $\left\lbrace v_1,\ldots, v_{h(j)}, Xv_j\right\rbrace $  vanish and hence the sets 
\begin{equation}\label{eq2.5}
\qquad B_{j,k}=\left\lbrace v_1, \ldots v_{h(j)}, Xv_j, v_{h(j)+1},\ldots \widehat{v_{k}}, \ldots v_n\right\rbrace 
\end{equation}
are linearly dependent, where $v_1, \cdots, v_n$ are the column vectors of the matrix $g$ 
and $\widehat{v_{k}}$ means removing $v_k$. Reciprocally, if
 for all $1\leq j\leq n-1$ and for all $k$ such that $h(j)+1\leq k\leq n$ all sets \eqref{eq2.5} are linearly dependent
then \eqref{eq2.4} is consequence of the following elementary lemma:
\begin{lem}\label{lem1.1}
Let $V$ be a complex vector space of finite dimension  $n$,  $B=\{v_1,\ldots,v_n\}$ a basis of $V$ and $v\in V$ non zero such that $v\neq v_j$ for all $1\leq j\leq n$. Then, for a fixed $j$,  $v\in V_j=\left\langle v_1,\cdots, v_j\right\rangle$ if and only if for all $k$ with $j+1\leq k\leq n$ the sets $B_{j,v,k}=(B- \{v_k\})\cup \{v\}$ are linearly dependent.\qed
\end{lem}

Now, we note that the conditions \eqref{eq2.3}, \eqref{eq2.4} and \eqref{eq2.5} have associated 
equations that arise from certain determinants: Indeed,
let $u_1,\cdots, u_n$ denote the column vectors of a generic square $n\times n$ matrix $Z=(z_{ij})$ and for a linear operator $X:{\mathbb C}^n\rightarrow {\mathbb C}^n$ consider the column vectors $Xu_i$, for $1\leq i\leq n$. For a positive integer $m\leq n$
let $R=(R_1,\ldots,R_m)$ with $R_i\in \{1,\ldots,n\}$ and $R_1<\cdots<R_m$ and let
$C=\{C_1,\ldots,C_m\}$ an ordered set of column vectors in $\k^{n}$. Let
$d_{R,C}\in{\mathbb C}[z_{ij}]$ denote the $m\times m$ minor obtained by choosing the rows of the columns of $C$ according to $R$, that is, the determinant of the $m\times m$ matrix whose  $(i,j)$-th entry is the $R_i$-th entry of $C_j$ . Now, for a Hessenberg variety $\hes(X,h)$ associated to a Hessenberg function $h$ and an endomorphism $X\in\fgl_n$
consider the following ideals
\begin{align*}
J_{X,j,h(j)}&=\left\langle d_{R,C} :R\subset \left\lbrace 1, 2, \ldots n\right\rbrace \text{ and }C\subset \left\lbrace u_1, \ldots u_{h(j)}, Xu_1,\ldots, Xu_j\right\rbrace\right. \\
&\qquad\qquad\left.\text{such that }\;  |R|=|C|=h(j)+1\right\rangle,\\
J'_{X,j,h(j)+1}&=\left\langle d_{R,C}: R=\{1,\ldots,h(j)+1\}\text{, }C=\{u_1,\ldots u_{h(j)},Xu_j\}\right\rangle,
\end{align*}
and note that:
\begin{itemize}
    \item Condition \eqref{eq2.3} has associated the following ideal
    $$J_{X,h}=\sum_{j=1}^{n} J_{X,j,h(j)}, $$
    \item condition \eqref{eq2.4} has associated the following ideal
    $$J'_{X,h}=\sum_{j=1}^{n} J'_{X,j,h(j)+1}$$
    \item and condition \eqref{eq2.5} has associated the following ideal 
    $$I_{X,H_h}=\left\langle d_{(1,\ldots, n),(u_1,\ldots,u_{k-1},Xu_j,u_{k+1},\ldots, u_n)}:1\leq j\leq n \text{ y }  k>h(j)\right\rangle.$$ 
\end{itemize}
Finally,   by the equivalence of Conditions \eqref{eq2.3} and \eqref{eq2.4}, the ideals $J_{X,h}=J'_{X,h}$, and 
since $\hes(X,h)$ is a closed subvariety \cite{Proc} of the complete flag variety, by Lemma \ref{lem1.1} $J'_{X,h}$ and $I_{X,H_h}$  are ideals of $\k[z_{ij},d^{-1}]$, where $d$ is the determinant function.  
Since $J_{X,h}$ is the ideal defining $\hes(X,h)$ and $I_{X,H_h}$ is the defining ideal of $\hes(X,H_h)$, as ideals in $\k[z_{ij},d^{-1}]$,
the equivalences $\eqref{eq2.3}\Leftrightarrow \eqref{eq2.4} \Leftrightarrow \eqref{eq2.5}$ and Lemma \ref{lem1.1} 
gives another proof of:

\begin{thm}[{\cite[Theorem 10]{Insko}}]\label{thm1.1}
For any Hessenberg function $h$
$$I_{X,H_h}=J_{X,h}$$
as ideals in $\k[z_{ij},d^{-1}]$.
\qed
\end{thm}

\section{One-dimensional orbits in a Hessenberg variety. The semisimple case}\label{sec3}
In this section we will give the one-dimensional orbits from a $T$-action on $\hes(X,h)$ in the case when $X$ is semisimple ($X\in \fft$) say $X=\dg(\lambda_1, \cdots,\lambda_n)$ and start with the example $h(j)=j+1$. Let $B\subseteq \GL_n$ the subgroup of upper triangular matrices (a Borel subgroup), $T\subset B$ an $n$-torus (diagonal matrices) and $\fgl$, $\ffb$ and $\fft$ the Lie algebras corresponding. Thus, $T=\left\lbrace s=\dg(T_1,\cdots, T_n): T_i\in \k^{*}\right\rbrace$ and $\fft=\left\lbrace \sigma=\dg(t_1,\cdots, t_n): t_i\in \k\right\rbrace$.  Consider the $T$-action on $\GL_n(\k)/B$ defined by left multiplication  $s\cdot[g]=[sg]$.  By Proposition \ref{prop1} the $T$-fixed points under the action are $[w]\in \GL_n(\k)/B$ such that $w\in S_n$. 
We will show that $\hes(X,h)$ is invariant under the $T$-action and once this is proven it determines which $T$-fixed points of $\GL_n(\k)$ belong to $\hes(X,h)$.
\begin{prop}\label{prop3.1}
Let $T$ be an $n$-torus acting on $\GL_n(\k)/B$ by left multiplication. Let $\hes(X,h)$ be a Hessenberg variety with $X\in \fg$ semisimple and $h(j)=j+1$. Then $\hes(X,h)$ is invariant under the $T$-action and $(\GL_n(\k)/B)^T=(\hes(X,h))^T$.
\end{prop}
\begin{proof}
We observe since $X$ is diagonal and $T$ is given by invertible diagonal matrices, $X$ commutes with all $s\in T$ and hence the $T$-action on $G/B$ induces an action on $\hes(X,h)$. Additionally, the $T$-fixed points are the same, $\hes(X,h)^T=(G/B)^T$. Indeed,
by Proposition \ref{prop1} the $T$-fixed points of $G/B$ are the $[w]\in G/B$ such that $w\in S_n$. On the other hand
$Xv_j=\lambda_{w(l)}v_j$ with $w(l)=j$, and thus
$$\left\lbrace v_1, \ldots, v_j, Xv_j,v_{j+1},\ldots, \widehat{v_k}, \ldots, v_n\right\rbrace $$
is linearly dependent for all $k$ with $j+1\leq k\leq n$. Therefore, by Theorem \ref{thm1.1},  $[w]\in \hes(X,h)$.
\end{proof}

The next theorem gives a description of the one-dimensional orbits of $\hes(X,h)$ under the action of $T$ given by Proposition \ref{prop3.1}. This description is analogous to the description of the one-dimensional orbits on the full flag variety cited in
Proposition \ref{prop1}. 
Let $O_{w,s_{j,k}w}$ denote the one-dimensional orbit of the full flag variety whose closure contains the $T$-fixed points $[w]$ and $[s_{j,k}w]$. This orbit consists of the flags
\begin{equation}\label{O1}
  O_{w,s_{j,k}w}=\left\lbrace [G_{jk}(c)w]\in \GL_n(\k)/B: c\in \k^*\right\rbrace  
\end{equation} 
where $[G_{jk}(c)w]$ and $s_{j,k}$  were defined in Section \ref{sec2}.

\begin{thm}\label{teo3.1}
Let $O_{w,s_{j,k}w}=\{[G_{jk}(c)w]:c\in{\k}^*\}$ be a one-orbit under the action of $T$ on $G/B$ as before. If $w(w^{-1}(k) +1)=j$, then  $O_{w,s_{j,k}w}\cap \hes(X,h)\neq \emptyset$. Moreover, for the $T$-action on $\hes(X,h)$ given by the restriction of the action of $T$ on $G/B$ as in \emph{Proposition \ref{prop3.1}}, then $O_{w,s_{j,k}w}\subset \hes(X,h)$.
\end{thm}
\begin{proof}
For each flag $[G_{jk}(c)w]$ with representative in $G/B$ given by the matrix
\begin{equation}\label{eq3.2}
{(G_{jk}(c)w)}_{i\ell}=
\begin{cases}
1 & \text{ if } w(\ell)=i, \\
c & \text{ if } i=j \text{ and } w^{-1}(k)=\ell,\\
0 & \text{ otherwise,}
\end{cases}
\end{equation}
for its column $v_{\ell}$, 
 with $\ell\neq w^{-1}(k)$, we have $Xv_{\ell}=\lambda_{w(\ell)}v_{\ell}$. It follows that 
\begin{equation}\label{eq4}
\left\lbrace v_1, \ldots v_{\ell},v_{\ell +1}, Xv_{\ell},\ldots \widehat{v_{r}}, \ldots v_n\right\rbrace 
\end{equation}
is a linearly dependent set. In particular, for each $r>h(\ell)$, it follows  that
\begin{equation*}
    d_{(1,\ldots, n),(v_1, \ldots v_{h(\ell)}, Xv_{\ell},\ldots \widehat{v_{r}}, \ldots v_n)}=0.
\end{equation*}
We only have to analyze the column vector $v_{w^{-1}(k)}$ where $a_{jw^{-1}(k)}=c$ and $a_{kw^{-1}(k)}=1$.  In these cases we have that
\begin{equation*}
    Xv_{w^{-1}(k)}=(0,\ldots, \lambda_{w^{-1}(k)}c,0, \ldots, \lambda_k, 0,\ldots, 0).
\end{equation*}

By hypothesis we have $w(w^{-1}(k) +1)=j$, and this implies  that $v_{w^{-1}(k)+1}$ has  the entry $a_{jw^{-1}(k) +1}=1$. Since $h(\ell)=\ell +1$, we have for all $r$ such that 
$r>h(w^{-1}(k))=w^{-1}(k)+1$, the determinant
\begin{equation}\label{eq5}
\tilde{d}=d_{(1,\ldots, n),(v_1, \ldots v_{w^{-1}(k)},v_{w^{-1}(k)+1}, Xv_{w^{-1}(k)},\ldots \widehat{v_{r}}, \ldots v_n)}=0.
\end{equation}
Now, since the column vector $v_r$ has entry $a_{\ell r}=1$ with $\ell=w(r)$ and $\ell\neq j$ so by substituting $v_r$ with $Xv_{w^{-1}(k)}$, the $\ell$ row of $\tilde{d}$ consists of zeros. We have shown that the column vectors of $[G_{jk}(c)w]$ satisfy the condition of the defining ideal  of  $\hes(X,h)$, and thus $[G_{jk}(c)w]\in \hes(X,h)$ for all $c\in \k^*$. It follows that  $O_{w,s_{j,k}w}\subset \hes(X,h)$.
\end{proof}
\begin{rem}\label{rem3.1}
Let $v_r$ be a column vector such that $a_{jr}=1$ ($w(r)=j$). If $r>w^{-1}(k)+1$, then there is $c\in \k^*$ such that
\begin{equation}\label{eq6}
\tilde{d}=d_{(1,\ldots, n),(v_1, \ldots v_{w^{-1}(k)},v_{w^{-1}(k)+1}, Xv_{w^{-1}(k)},\ldots \widehat{v_{r}}, \ldots v_n)}\neq 0,
\end{equation}
this is because 
\begin{equation}
    \tilde{d}=(-1)^{r+j}c\sgn(\eta)+(-1)^{r+k}\sgn(\tau\eta)
\end{equation}
where $\tau$, $\eta$ are in $S_{n-1}$ and $\tau$ is a transposition. Thus $[G_{jk}(c)w]\notin \hes(X,h)$.
\end{rem}
The next theorem is a generalization of Theorem \ref{teo3.1} for an arbitrary Hessenberg function. Theorem \ref{teo2} allow us to determine which one-dimensional orbits of $\GL_n(\k)/B$ are contained in  the corresponding Hessenberg variety:
\begin{thm}\label{teo2}
Let $T$ be an $n$-torus acting on $\GL_n(\k)/B$ by   left multiplication. Let $\hes(X,h)$ be a Hessenberg variety with $X\in \fg$ semisimple and $h$ any Hessenberg function. Then 
\begin{enumerate}
    \item $\hes(X,h)$ is invariant under the $T$-action.
    \item $(\GL_n(\k)/B)^T=(\hes(X,h))^T$.
    \item If $w(\ell)=j$, then $\ell\leq h(w^{-1}(k))$ if and only if for all $c\in \k^*$, $[G_{jk}(c)w]\in \hes(X,h)$, that is
$$O_{w,s_{j,k}w}\subset \hes(X,h).$$
\end{enumerate}
\end{thm}
\begin{proof} Parts (1) and (2) are proved as in Proposition \ref{prop3.1}. For Part (3), if 
$\ell\neq w^{-1}(k)$, $Xv_{\ell}=\lambda_{w(\ell)}$,  by a similar argument as in the proof of Theorem \ref{teo3.1}, we have the linear dependence condition of \eqref{eq4}, thus, we have
\begin{equation*}
  d_{(1,\ldots, n),(v_1, \ldots v_{h(\ell)}, Xv_{\ell},\ldots \widehat{v_{r}}, \ldots v_n)}=0.  
\end{equation*}
Again, we have to analyze $v_{w^{-1}(k)}$, where $Xv_{w^{-1}(k)}=(0,\ldots, \lambda_j c,\ldots, \lambda_k,\ldots, 0)$. Then, for all $r>h(w^{-1}(k))$ we have
\begin{equation}\label{eq7}
\tilde{d}=d_{(1,\ldots, n),(v_1, \ldots v_{w^{-1}(k)},\ldots,v_{h(w^{-1}(k))}, Xv_{w^{-1}(k)},\ldots \widehat{v_{r}}, \ldots v_n)}=0,
\end{equation}
where since $v_r$  has entry $a_{w(r)r}=1$ and $a_{ir}=0$ for $i\neq w(r)$, particularly, $a_{jr}=a_{w(\ell)r}=0$, by hypothesis, $w(r)\neq w(\ell)=j\leq h(w^{-1}(k))$. Then when we substitute $v_r$ by $Xv_{w^{-1}(k)}$, $\tilde{d}$ has zeros in the row $w(r)$ because the entry $(Xv_{w^{-1}(k)})_{w(r)r}=0$ for all $r>h(w^{-1}(k))$. Thus $$[G_{jk}(c)w]\in \hes(X,h)$$
for all $c\in \k^*$.  

Assume now that $O_{w,s_{j,k}w}\subset \hes(X,h)$, that is, $[G_{jk}(c)w]\in \hes(X,h)$ for all $c\in \k^*$. Let $\{v_1,\ldots, v_n\}$ be the column vectors of   the matrix $G_{jk}(c)w$  which represents the flag $[G_{jk}(c)w]\in G/B$. Then, the ideal $I_{X,H_h}$ of the Theorem \ref{thm1.1} vanishes on the columns $v_{\ell}$.
Now, for each $v_{\ell}$ with $\ell\neq w^{-1}(k)$ and $Xv_{\ell}=\lambda_{w(\ell)}v_{\ell}$  the following  determinants vanish
\begin{equation*}
    \tilde{d}=d_{(1,\ldots, n),(v_1, \ldots v_{\ell},\ldots,v_{h(\ell)}, Xv_{\ell},\ldots \widehat{v_{r}}, \ldots v_n)}=0
\end{equation*}
and for $\ell=w^{-1}(k)$, we have $Xv_{w^{-1}(k)}=\lambda_j c v_{w^{-1}(j)} + \lambda_k v_{w^{-1}(k)}$. Then, by hypothesis and by Theorem \ref{thm1.1},  for all $r>h(w^{-1}(k))$ we have the first equality in:
\begin{align}
\begin{aligned}
0 &= \tilde{d}=d_{(1,\ldots, n),(v_1, \ldots v_{w^{-1}(k)},\ldots,v_{h(w^{-1}(k))}, Xv_{w^{-1}(k)},\ldots \widehat{v_{r}}, \ldots v_n)}\\
&=  d_{(1,\ldots, n),(v_1, \ldots v_{w^{-1}(k)},\ldots,v_{h(w^{-1}(k))},\lambda_j c v_{w^{-1}(j)} + \lambda_k v_{w^{-1}(k)} ,\ldots \widehat{v_{r}}, \ldots v_n)}\\
&= d_{(1,\ldots, n),(v_1, \ldots v_{w^{-1}(k)},\ldots,v_{h(w^{-1}(k))},\lambda_j c v_{w^{-1}(j)} ,\ldots \widehat{v_{r}}, \ldots v_n)}\\
&\quad +d_{(1,\ldots, n),(v_1, \ldots v_{w^{-1}(k)},\ldots,v_{h(w^{-1}(k))}, \lambda_k v_{w^{-1}(k)} ,\ldots \widehat{v_{r}}, \ldots v_n)}\\
& =d_{(1,\ldots, n),(v_1, \ldots v_{w^{-1}(k)},\ldots,v_{h(w^{-1}(k))},\lambda_j c v_{w^{-1}(j)} ,\ldots \widehat{v_{r}}, \ldots v_n)}.
\end{aligned}
\end{align}
The last equality implies that 
$$B_{w^{-1}(j)}=\{v_1, \ldots v_{w^{-1}(k)},\ldots,v_{h(w^{-1}(k))},\lambda_j c v_{w^{-1}(j)} ,\ldots \widehat{v_{r}}, \ldots v_n)\}$$
is a linearly dependent set for all $r>h(w^{-1}(k))$. This linear dependence condition occurs if $v_{w^{-1}(j)}$ and $\lambda_j c v_{w^{-1}(j)}$ belong to $B_{w^{-1}(j)}$ for all $r>h(w^{-1}(k))$. This happens whenever $w^{-1}(j)\leq h(w^{-1}(k))$ by Lemma \ref{lem1.1}. Thus, if $O_{w,s_{j,k}w}\subset \hes(X,h)$ then  $\ell\leq h(w^{-1}(k))$, this proves Part (3).
\end{proof}
\begin{expl}
We consider the full flag variety $\GL_n(\k)/B$ and $T$ an $n$-torus acting on $\GL_n(\k)$ as Proposition \ref{prop1}. Let $X\in \fg$ be semisimple and $h$ Hessenberg function defined by $h(i)=n$ for all $1\leq i\leq n$. For all one-dimensional orbit $O_{w,s_{j,k}w}$ of $\GL_n(\k)/B$, we have $\overline{O_{w,s_{j,k}w}}\subset \hes(X,h)$. Indeed, for all pair $j$, $k$ with $j<k$ and $w^{-1}(k)<w^{-1}(j)$, the condition (3) of \ref{teo2} holds since $w^{-1}(j)\leq n=h(w^{-1}(k))$. Particularly, when $X=Id$, we have Proposition \ref{prop1}. 
\end{expl}
\begin{rem}\label{rem3.2} 
 If $h$ is the Hessenberg function $h(i)=i$ for all $i$ and $X\in \fg$ is diagonal, for the corresponding
 Hessenberg variety $\hes(X,h)\subset \GL_n(\k)/B$, consider an $n$-torus $T$ acting on $\GL_n(\k)/B$ by left multiplication. By Proposition \ref{prop3.1} $\hes(X,h)$ is invariant under this $T$-action. However, for every one-dimensional orbit $O_{w,s_{j,k}w}$ of $\GL_n(\k)/B$  with $w(\ell)=j$, by construction $\ell > w^{-1}(k)$, and  since $h(i)=i$ then $\ell>h(w^{-1}(k))$. By a similar argument as in Remark \ref{rem3.1} there is $c$ such that
\begin{equation}
\tilde{d}=d_{(1,\ldots, n),(v_1, \ldots, v_{w^{-1}(k)},\ldots,v_{h(w^{-1}(k))}, Xv_{w^{-1}(k)},\ldots \widehat{v_{\ell}}, \ldots v_n)}\neq 0,
\end{equation}
hence $[G_{jk}(c)w]\notin \hes(X,h)$ and
thus the $T$-action on $\hes(X,h)$ does not have one-dimensional orbits.
\end{rem}
\section{One dimensional orbits in a Hessenberg variety. The nilpotent case}\label{sec4}
In this section we obtain the one-dimensional orbits from a $T$-action on the Hessenberg variety $\hes(N,h)$ when $N$ is a regular nilpotent operator. Thus,  $N$ is a matrix whose Jordan form has one block with corresponding eigenvalue equal to zero. We obtain an analogue description as in the diagonal case in Theorem \ref{teo2}. Here, $T$ is an $n$-torus acting on flag variety $\GL_n(\k)/B$ by left multiplication. First, we determine which fixed points in $(\GL_n(\k)/B)^T$ belong to $\hes(N,h)$. Next, we do same for the one dimensional orbits of $\GL_n(\k)/B$.

By Proposition \ref{prop1}, the $T$-fixed points  $[w]\in \GL_n(\k)/B$ correspond to $w\in S_n$. For $[w]\in (\GL_n(\k)/B)^T$ with column vectors $v_1,\ldots, v_n$, $Nv_{w^{-1}(i+1)}=e_{i}$ for
all $1\leq i < n$,  where $e_i$ are the canonical base, and for $i=1$, $Nv_{w^{-1}(1)}=Ne_1=\overline{0}$, the null vector. By Theorem \ref{thm1.1}, a
$T$-fixed point $[w]$ belongs to $\hes(N,h)$ if and only if its column vectors vanish in the ideal $I_{N,H_h}$. This is equivalent to $w^{-1}(i)\leq h(w^{-1}(i+1))$ for all $1\leq i<n$. This last claim follows from the fact that $Nv_{w^{-1}(i+1)}$ and $v_{w^{-1}(i+1)}$ are linearly independent since
\begin{equation*}
   \tilde{d}=d_{(1,\ldots, n),(v_1, \ldots v_{w^{-1}(i+1)},\ldots,v_{h({w^{-1}(i+1)})}, Nv_{{w^{-1}(i+1)}},\ldots \widehat{v_{r}}, \ldots v_n)}=0 
\end{equation*}
if and only if $e_i=v_{w^{-1}(i)}$ belongs to
\begin{equation}\label{eq4.1}
B_{w^{-1}(i+1),r}=\left\lbrace v_1, \ldots v_{h(w^{-1}(i+1))}, Nv_{w^{-1}(i+1)},\ldots \widehat{v_{r}}, \ldots v_n\right\rbrace
\end{equation}
for all $r>h(w^{-1}(i+1))$, and \eqref{eq4.1} is equivalent to $w^{-1}(i)\leq h(w^{-1}(i+1))$.

Now, we compute the one dimensional orbits of $\hes(N,h)$. From Proposition \ref{prop3.1} and Equation \eqref{eq3.2}  the one dimensional orbits of $\GL_{n}(\k)/B$ are $O_{w,s_{j,k}w}$ with $w\in S_n$ and  the transposition $s_{j,k}$. For $w\in S_n$ with $[w]\in \GL_n/B$ the corresponding flag, in \eqref{O1} we defined
\begin{equation}\label{eq4.2}
  O_{w,s_{j,k}w}=\left\lbrace [G_{jk}(c)w]\in \GL_n(\k)/B: c\in \k^*\right\rbrace  
\end{equation} 
with $j<k$ such that $w^{-1}(j)>w^{-1}(k)$.
To determine which of the orbits \eqref{eq4.2} belong to $\hes(N,h)$, first note that $[w]$ must belong to $\hes(N,h)$ and if that is the case, $O_{w,s_{j,k}} \subset \hes(N,h)$ if every flag $[G_{j,k}(c)w]$ of $O_{w,s_{j,k}w}$ vanish in the ideal $I_{N,H_h}$ and the proof of this last claim starts by noting that 
\begin{equation}\label{eq4.3.1}
Nv_{w^{-1}(k)}=
\begin{cases}
e_{k-1} + ce_{j-1} & \text{ if } j,k\neq 1, \\
e_{k-1} & \text{ if } j=1,\;\text{note that $k>j=1$}
\end{cases}
\end{equation}
and for $j,k\neq 1$ in \eqref{eq4.3.1}, $[G_{j,k}(c)w]$ vanish in $I_{N,H_h}$ if and only if for all $r>h(w^{-1}(k))$ we have the first equality in \eqref{eq4.2.1}, where the next equalities follow from a direct computation:
\begin{align}\label{eq4.2.1}
\begin{aligned}
0 &= \tilde{d}=d_{(1,\ldots, n),(v_1, \ldots v_{w^{-1}(k)},\ldots,v_{h(w^{-1}(k))}, Nv_{w^{-1}(k)},\ldots \widehat{v_{r}}, \ldots v_n)}\\
&=  d_{(1,\ldots, n),(v_1, \ldots v_{w^{-1}(k)},\ldots,v_{h(w^{-1}(k))},e_{k-1} + ce_{j-1} + \lambda_k v_{w^{-1}(k)} ,\ldots \widehat{v_{r}}, \ldots v_n)}\\
&= d_{(1,\ldots, n),(v_1, \ldots v_{w^{-1}(k)},\ldots,v_{h(w^{-1}(k))},e_{k-1} ,\ldots \widehat{v_{r}}, \ldots v_n)}\\
&\quad +d_{(1,\ldots, n),(v_1, \ldots v_{w^{-1}(k)},\ldots,v_{h(w^{-1}(k))}, ce_{j-1} ,\ldots \widehat{v_{r}}, \ldots v_n)}.
\end{aligned}
\end{align}
We observe now  the right-hand side of  the last equality in \eqref{eq4.2.1} is zero if and only if for all $r>h(w^{-1}(k))$ both determinants are simultaneously zero, that is
\begin{align}\label{eq4.3}
\begin{aligned}
0&= d_{(1,\ldots, n),(v_1, \ldots v_{w^{-1}(k)},\ldots,v_{h(w^{-1}(k))},e_{k-1} ,\ldots \widehat{v_{r}}, \ldots v_n)}\\
&= d_{(1,\ldots, n),(v_1, \ldots v_{w^{-1}(k)},\ldots,v_{h(w^{-1}(k))}, ce_{j-1} ,\ldots \widehat{v_{r}}, \ldots v_n)}.
\end{aligned}
\end{align}
Indeed, by Lemma \ref{lem1.1} there is $r'>h(w^{-1}(k))$ such that the two determinants in the last equality in \eqref{eq4.2.1} are nonzero and cancel each other if and only if $w^{-1}(k-1)$ and $w^{-1}(j-1)$ are greater than $h(w^{-1}(k))$. 
Now, for $r= w^{-1}(k-1)$ we have
\begin{align}\label{eq4.4}
\begin{aligned}
0&\neq d_{(1,\ldots, n),(v_1, \ldots v_{w^{-1}(k)},\ldots,v_{h(w^{-1}(k))},e_{k-1} ,\ldots \widehat{v_{r}}, \ldots v_n)}\\
\end{aligned}
\end{align}
and hence
\begin{align}\label{eq4.5}
\begin{aligned}
&0= d_{(1,\ldots, n),(v_1, \ldots v_{w^{-1}(k)},\ldots,v_{h(w^{-1}(k))}, ce_{j-1} ,\ldots \widehat{v_{r}}, \ldots v_n)}.
\end{aligned}
\end{align}
Similarly, for $r=w^{-1}(j-1)$, Equation \ref{eq4.5}
is nonzero and Equation \ref{eq4.4} is zero. Thus, both determinants cannot be nonzero simultaneously proving \eqref{eq4.3}. 
Lastly, the vanishing of both summands in \eqref{eq4.3} holds if and only if
$$ w^{-1}(k-1)\leq h(w^{-1}(k)) \qquad\text{ and }\qquad w^{-1}(j-1)\leq h(w^{-1}(k)).$$

Next, for the case $j=1$ in \eqref{eq4.3.1} equality \eqref{eq4.6} holds
\begin{align}\label{eq4.6}
\begin{aligned}
0&= d_{(1,\ldots, n),(v_1, \ldots v_{w^{-1}(k)},\ldots,v_{h(w^{-1}(k))},e_{k-1} ,\ldots \widehat{v_{r}}, \ldots v_n)}
\end{aligned}
\end{align}
if and only if $w^{-1}(k-1)\leq h(w^{-1}(k))$.

These arguments characterize when an open orbit $O_{w,s_{j,k}w}$ is contained in $\hes(N,h)$. To determine when the
 closure of a one dimensional orbit is also contained in $\hes(N,h)$ we need to characterize when $[s_{j,k}w]\in \hes(N,h)$. For this, observe that $[s_{j,k}w]$ is the permutation matrix that interchanges the rows  $j$ and $k$ of $[w]$. In other words, $[w^{-1}s_{j,k}]$ interchanges the columns $w^{-1}(j)$ and $w^{-1}(k)$ of $[w^{-1}]$ leaving the remaining columns remain fixed. Hence, it suffices to verify  that the columns $w^{-1}(j)$ and $w^{-1}(k)$ vanish in $I_{N,H_h}$. For the column  $w^{-1}(k)$ this is equivalent to show that $(s_{j,k}w)^{-1}(j)\leq h(w^{-1}(j+1))$. To check this, observe that for $k\neq j+1$,
\begin{align}
\begin{aligned}
(s_{jk}w)^{-1}(j) &= (w^{-1}s_{jk})(j)=w^{-1}(k)\leq h((s_{jk}w)^{-1}(j+1))=h(w^{-1}(j+1))\\
\end{aligned}
\end{align}
and for $k=j+1$, we have $w^{-1}(j+1)\leq h(w^{-1}(j))$.
Similarly, for the column $w^{-1}(j)$,
\begin{align}
\begin{aligned}
(s_{jk}w)^{-1}(k) &= (w^{-1}s_{jk})(k)=w^{-1}(j)\leq h((s_{jk}w)^{-1}(k+1))=h(w^{-1}(k+1)).\\
\end{aligned}
\end{align}

Hence, $\overline{O_{w,s_{j,k}w}}\subset \hes(N,h)$ if and only if
\begin{enumerate}
    \item $[w]\in \hes(N,h)^T$
    \item $[s_{jk}w] \in \hes(N,h)^T$
    \item $w^{-1}(k-1)\leq h(w^{-1}(k)$ and $w^{-1}(j-1)\leq h(w^{-1}(k))$
\end{enumerate}
and we have proved the following:
\begin{thm}\label{teo4}
Let $T$ be an $n$-torus acting on $\GL_n(\k)/B$ by left multiplication. Let $\hes(N,h)$ be a Hessenberg variety with $N\in \fg$ regular nilpotent operator and $h$ any Hessenberg function. Then $\overline{O_{w,s_{j,k}w}}\subset \hes(N,h)$ if and only if
\begin{enumerate}
    \item $w^{-1}(i)\leq h(w^{-1}(i+1))$ for all $i$ such that $1\leq i <n$.\label{inc4.1}
    \item $w^{-1}(k)\leq h(w^{-1}(j+1))$ for $k\neq j+1$ and for $k=j+1$ we have $w^{-1}(j+1)\leq h(w^{-1}(j))$. It must also be satisfied  $w^{-1}(j)\leq h(w^{-1}(k+1))$ if $k<n$.
    \item $w^{-1}(k-1)\leq h(w^{-1}(k))$ and $w^{-1}(j-1)\leq h(w^{-1}(k))$.
\end{enumerate}
\end{thm}
\begin{expl}\label{ex4.1}
For $\GL_3(\k)/B$ and $h(i)=i+1$, consider the nilpotent operator
$$N=
\begin{pmatrix}
   0  & 1  & 0 \\
   0  & 0 & 1 \\
   0  & 0 & 0
\end{pmatrix}$$
Let $T$ be the $3$-torus in $\GL_3(\k)$ and consider its action on $GL_3(\k)/B$ by left multiplication . The $T$-fixed points of $\GL_3(\k)/B$ are 
\begin{align}
\begin{aligned}
\left[ e\right]=\left(
\begin{array}{ccc}
   1  & 0  & 0 \\
   0  & 1 & 0 \\
   0  & 0 & 1
\end{array}
\right ), &
\left[ w_1\right]=\left(
\begin{array}{ccc}
   1  &  0 & 0 \\
   0  & 0 & 1 \\
   0  & 1 & 0
\end{array}
\right ), &
\left[ w_2\right]=\left(
\begin{array}{ccc}
   0  & 1  & 0 \\
   1  & 0 & 0 \\
   0  & 0 & 1
\end{array}
\right ),\\
\left[ w_3\right]=\left(
\begin{array}{ccc}
   0  & 0  & 1 \\
   1  & 0 & 0 \\
   0  & 1 & 0
\end{array}
\right ),&
\left[ w_4\right]=\left(
\begin{array}{ccc}
   0  & 1  & 0 \\
   0  & 0 & 1 \\
   1 & 0 & 0
\end{array}
\right ),&
\left[ w_5\right]=\left(
\begin{array}{ccc}
   0  & 0  & 1 \\
   0  & 1 & 0 \\
   1  & 0 & 0
\end{array}
\right ),
\end{aligned}
\end{align}
By Theorem \ref{teo4} we find the $T$-fixed points and one dimensional orbits of $\hes(N,h)$.  The identity $e$ belong to $\hes(N,h)$ and
\begin{enumerate}
    \item for $[w_1]$ we have 
    \begin{itemize}
        \item $w_{1}^{-1}(1)=1\leq h(w_{1}^{-1}(2))=h(3)=3$,
        \item $w_{1}^{-1}(2)=3= h(w_{1}^{-1}(3))=h(2)=3$,
    \end{itemize}
     then $[w_1]\in \hes(N,h)$
    \item for $[w_2]$ we have
    \begin{itemize}
        \item $w_{2}^{-1}(1)=2=h(w_{2}^{-1}(2))=h(1)=2$,
        \item $w_{2}^{-1}(2)=1\leq h(w_{2}^{-1}(3))=h(3)=3$,
    \end{itemize}
    then $[w_2]\in \hes(N,h)$
    \item for $[w_3]$ we have 
    \begin{itemize}
        \item $w_{3}^{-1}(1)=3> h(w_{3}^{-1}(2))=h(1)=2$,
    \end{itemize}
     then $[w_3]\notin\hes(N,h)$
     \item for $[w_4]$ we have 
    \begin{itemize}
        \item $w_{4}^{-1}(2)=3> h(w_{4}^{-1}(3))=h(1)=2$,
    \end{itemize}
     then $[w_4]\notin \hes(N,h)$
     \item for $[w_5]$ we have 
    \begin{itemize}
        \item $w_{5}^{-1}(1)=3\leq h(w_{5}^{-1}(2))=h(2)=3$,
        \item $w_{5}^{-1}(2)=2= h(w_{5}^{-1}(3))=h(1)=2$,
    \end{itemize}
     then $[w_5]\in \hes(N,h)$.
     \end{enumerate}
    Thus, the $T$-fixed points of  $\hes(N,h)$ are
    $\{ e, [w_{1}], [w_{2}], [w_{5}]\}$.
    
By Theorem $\ref{teo4}$,  $\overline{O_{w,s_{j,k}w}}$ is contained in $\hes(N,h)$ for $[w]$, $[s_{j,k}w]$ fixed points. Then, we compute these orbits for  $\{ e, [w_{1}], [w_{2}], [w_{5}]\}$. We observe for $e$ there are not pairs $(j,k)$. Now:
\begin{enumerate}
    \item For $[w_1]$ it has the pair $(2,3)$ and $[s_{2,3}w_{1}]$. Then
    \begin{itemize}
        \item $w^{-1}(3)=2\leq h(w^{-1}(3))=h(2)=3$
        \item $w^{-1}(2)=3= h(w^{-1}(2))=3$
    \end{itemize}
    hence $[s_{2,3}w_1]\in \hes(N,h)$. Now we verify condition (3) of Theorem \ref{teo4} for $O_{w_1 s_{2,3}w_1}$
    \begin{itemize}
        \item $w_1^{-1}(2)=3= h(w_1^{-1}(3))=h(2)=3$
        \item $w_1^{-1}(1)=1\leq h(w_1^{-1}(3))=3$
    \end{itemize}
    then $O_{w_1 s_{2,3}w_1}\subset \hes(N,h)$. Thus $\overline{O_{w_1,s_{2,3}w_1}}\subset \hes(N,h)$.
   \item For $[w_2]$ it has the pair $(1,2)$ and $[s_{1,2}w_{2}]$. Then
    \begin{itemize}
        \item $w_2^{-1}(2)=1\leq h(w_2^{-1}(1))=h(2)=3$
        \item $w_2^{-1}(1)=2= h(w_2^{-1}(3))=3$
    \end{itemize}
    hence $[s_{1,2}w_2]\in \hes(N,h)$. Now we verify condition (3) of Theorem \ref{teo4} for $O_{w_1 s_{2,3}w_1}$
    \begin{itemize}
        \item $w_2^{-1}(1)=2= h(w_2^{-1}(2))=h(1)=2$
    \end{itemize}
    then $O_{w_2 s_{1,2}w_2}\subset \hes(N,h)$. Thus $\overline{O_{w_2,s_{1,2}w_2}}\subset \hes(N,h)$.
    \item For $[w_5]$ we have the pairs $(1,2)$, $(1,3)$ and $(2,3)$. Then 
    \begin{itemize}
        \item for $[s_{1,2}w_5]$, $w_5^{-1}(1)=\;3> h(w_5^{-1}(3))=h(1)=2$ thus $[s_{1,2}w_5]\notin\hes(N,h)$.
        \item for $[s_{2,3}w_5]$, $w_5^{-1}(1)=3> h(w_5^{-1}(3))=2$ thus $[s_{2,3}w_5]\notin \hes(N,h)$
        \item for $[s_{1,3}w_5]$, $w_5^{-1}(2)=2=h(w_5^{-1}(3))=h(1)=2$. Thus $[s_{1,3}w_5]\in \hes(N,h)$.
    \end{itemize}
    hence $[s_{1,3}w_5]\in \hes(N,h)$. Finally we show  condition (3) of Theorem \ref{teo4} for $O_{w_1 s_{1,3}w_1}$
    \begin{itemize}
        \item $w_1^{-1}(2)=2\leq h(w_1^{-1}(3))=h(2)=3$
    \end{itemize}
    then $O_{w_5 s_{1,3}w_1}\subset \hes(N,h)$. Thus $\overline{O_{w_5,s_{1,3}w_1}}\subset \hes(N,h)$.
    \end{enumerate}

The corresponding \textit{moment graph} is represented by Figure \ref{fig 4.1}

\begin{figure}[h!]
    \begin{tikzpicture}[scale=0.8]
\node (1) at (0,0) {};  
\node[blue] at (-.5,0) {$w_1$};
 \node (2) at (1.5,-1.5) {};
 \node at (1.5,-1.8)[blue] {$e$};
 \node (3) at (3,0) {} ;
  \node[blue] at (3.5,0) {$w_2$};
 \node (4) at  (1.5,1.5) {};
 \node[blue] at (1.5,1.8) {$w_5$};
\draw (0,0) -- (1.5,-1.5) -- (3,0);
\draw (1.5,1.5) -- (1.5,-1.5);
\draw (0,0) node[red] {$\bullet$};
\draw (1.5,-1.5) node[red] {$\bullet$};
\draw (3,0) node[red] {$\bullet$};
\draw ((1.5,1.5) node[red] {$\bullet$};
\end{tikzpicture}
\caption{Moment graph of $\hes(N,h)$.}
\label{fig 4.1}
\end{figure}

\end{expl}

\begin{expl}\label{ex4.2}
We consider $\GL_n (\k)/B$ and let $T$ be  the $n$-torus acting on $\GL_n(\k)/B$ by left multiplication. Let $N$ be a regular nilpotent operator (say $N$ as in Example \ref{ex4.1} in the case $n=3$). We define $h(i)=i$ for all $1\leq i\leq n$ and we consider $\hes(N,h)$. We observe that $\hes(N,h)$ has only one fixed point, $[e]$. Indeed, for all $i$ we have $$w^{-1}(i) < w^{-1}(i+1)= h(w^{-1}(i+1)$$
 and this is true only for $[e]$.
\end{expl}

\begin{rem}\label{obs4.1}
We consider $h(i)=i$ for all $1\leq i \leq n$.
In the section \ref{sec3}, the Remark \ref{rem3.2} we analyze the case $\hes(X,h)$ with $X$ diagonal. Its fixed points are the same $\GL_n(\k)/B$. For $\hes(N,h)$ with $N$ nilpotent, has only one fixed point, $[e]$, as proved in Example \ref{ex4.2}.
\end{rem}

\section{The equivariant cohomology of a Hessenberg variety in the nilpotent case with an action of a 1-dimensional torus}\label{sec5}

In this section, we consider the action by left multiplication  of the one dimensional torus  $S=\lbrace \dg(t, t^2, \ldots, t^n):t\in \k^*\rbrace$ on $\hes(N,h)$. In the whole section, we will assume that the cohomology $H_{S}^*(\hes(N,h);\k)$ is generated by its two degree cohomology  classes (this last assumption is a conjecture).

First, consider the action of the torus on $\GL_n(\k)/B$ by left multiplication. For a regular nilpotent operator $N$ and $h$ any Hessenberg function the variety $\hes(N,h)$ is $S$-invariant and the $S$-fixed points of $\GL_{n}(\k)/B$ are given by the permutation matrices, $[w]\in \GL_n(\k)/B$ such that $w\in S_n$. The proof of part (1) of Theorem \ref{teo4} also gives which fixed points of the action of $S$ on $\GL_n(\k)/B$ are fixed points in the corresponding Hessenberg subvariety: 

\begin{cor}\label{cor4.5}
 Let $\hes(N,h)$ be the Hessenberg variety for $N\in \fg$ a regular nilpotent operator and $h$ any Hessenberg function and with the action of the one dimensional torus $S$. Then, for all $[w]\in (\GL_n(\k)/B)^S$, $[w]\in \hes(N,h)^S$ if and only if
 $w^{-1}(i)\leq h(w^{-1}(i+1))$ for all $i$ such that $1\leq i <n$.
 \end{cor}
 E. Insko \cite[Theorem 4.14]{phd} gives a description of the $S$-equivariant cohomology  class of any $\hes(N,h)$ in the $S$-equivariant cohomology ring of $\GL_n(\k)/B$:  
 
 \begin{thm}[Insko, Theorem 4.14]\label{InskoLoc}
 For  $wB=[w]$ an $S$-fixed points of $\GL_{n}(\k)/B$, if
$wB$ is a point in $\hes(N,h)$, the localization in the $S$-equivariant cohomology ring of $\GL_n(\k)/B$ of the equivariant cohomology class of $[\hes(N,h)]_{S}$ at $wB$ is
 $$\prod_{\lbrace1\leq i\leq k, 1\leq j\leq n:h(j)< w^{-1}(i)\rbrace} ((1-w(j)+1)t).$$
 \end{thm}
 With this result Insko computes the $S$-equivariant class of $\hes(N,h)$ in 
 $$H^{*}_{S}(\GL_n(\k)/B)^{S} \simeq \bigoplus_{w\in W}(\k[t])$$
 and proves that $[\hes(N,h)]_{S}$ is the tuple in $\oplus_{w\in W}(\k[t])$ consisting of the localizations at $wB=[w]$ for $[w]\in \hes(N,h)$ and zero elsewhere.
We can improve this result using our description of the $S$-fixed point of $\hes(N,h)$: 
 \begin{cor}\label{cor4.6}
Let $\hes(N,h)$ be a Hessenberg variety with the action the one-dimensional torus $S$ as before. Then, the localization of the non zero equivariant cohomology classes   $wB\in(\GL_n(\k)/B)^S$ are given by
 $$\prod_{\lbrace1\leq i\leq k, 1\leq j\leq n:h(j)< w^{-1}(i)\leq h(w^{-1}(i+1))\rbrace} ((1-w(j)+1)t).$$
 \end{cor}
The difference between Theorem \ref{InskoLoc} and Corollary \ref{cor4.6} is that the former needs to know which fixed points belong to $\hes(N,h)$ and the corollary does not since by Corollary \ref{cor4.5} the condition $w^{-1}(i)\leq h(w^{-1}(i+1))$ means that $[w]\in \hes(N,h)$.

Now, we give another description of the fixed points of $\hes(N,h)$. Condition \ref{inc4.1} of Theorem \ref{teo4} is equivalent to:
\begin{equation}\label{eq5.1}
wB\in \hes(N,h)^S \Leftrightarrow w(i)-1=w(k) \text{ with  } k\leq h(i), \text{ if } w(i)\neq 1.
\end{equation}
Equivalently, there is $k$ with $1 \leq k \leq h(i)$ such that $Nv_i=v_k$, that is, $w(i)-1=w(k)$. Hence, we must find an  ideal $\tilde{I}_h\subseteq \k[x_1,\ldots, x_n]$ such that its zero set $V(\tilde{I}_h)$ contains some permutations $w\in S_n$. The indeterminate $x_r$ will be interpreted as the value of the permutation $w$ in $r$. This represents the entry $w(r)$ of the column vector $v_r$.   Using the equivalence \eqref{eq5.1} we will prove that $\tilde{I}_h$ detects the fixed points of $\hes(N,h)$ for the $S$-action. Let $wB\in \hes(N,h)$ be a fixed point with column vectors $\{v_1, \ldots, v_n\}$. Since $wB$ satisfies the Insko's ideal of the Theorem \ref{thm1.1}, we can assume without loss of generality that for all $i$
\begin{align}\label{eq4.12}
\begin{aligned}
0&=d_{(1,\ldots, n),(v_1, \ldots v_{i},\ldots,v_{h(i)},Nv_i ,\ldots \widehat{v_{r}}, \ldots v_n)}
\end{aligned}
\end{align}
with $r=h(i)+1$. 
Starting with $v_1$,  if $v_1\neq e_1$ and $h\neq h_0$, where $h_0(i)=i$ for all $1\leq i\leq n$, we have
\begin{align}\label{cd4.1}
 Nv_1=v_{k_1}& \text{ with } w(1)-1=w(k_1)  \text{ and } 2\leq k_1\leq h(1).   
\end{align}
By \eqref{eq4.12}, Condition \eqref{cd4.1} is equivalence to considering the following determinant
\begin{align}\label{eq4.13}
\begin{aligned}
0&=d_{(1,\ldots, n),(v_1, v_2,\ldots,v_{h(1)},Nv_1 , \ldots v_n)}.
\end{aligned}
\end{align} 
Since $2\leq k_1 \leq h(1)$, for $v\neq e_1$ all solutions to conditions \eqref{cd4.1} and \eqref{eq4.13} are detected by the polynomial
\begin{align}
\begin{aligned}
\tilde{g}'_1&=(x_1-x_2 -1)\cdots (x_1-x_{h(1)} -1).
\end{aligned}
\end{align} 
Now, for $Nv_1=0$ or for $h_0$, both cases imply $v_1=e_1$. This condition is detected by the monomial $(x_1 -1)=0$. Hence we define
\begin{align}\label{eq4.14}
\begin{aligned}
\tilde{g}_1&=(x_1 -1)(x_1-x_2 -1)\cdots (x_1-x_{h(1)} -1).
\end{aligned}
\end{align}
For the next column vectors $v_i$ we do a similar analysis but now we need to consider the nesting condition also. This means, in addition to $Nv_i=v_{k_i}$, $1\leq k_i\leq h(i)$ we must also have $Nv_j=v_{\ell}$ with $1\leq \ell \leq h(i)$, $j\neq i$, and $j_{\ell}\leq j\leq i-1$ where $j_{\ell}$ is a largest integer such that $j_{\ell}<\ell$ , $h(j_{\ell}-1)<\ell$ and $\ell \leq h(j_{\ell})$. 
Now, observe that if $wB$ is such that $Nv_i=v_{k_i}$ with $j_{i}\leq k_i<i$, then the fixed point $\tau_{i, k_i}wB$ satisfies Condition \eqref{eq5.1} for the column vector $v_{k_i}$. In this case, it is enough to detect $wB$ or $\tau_{i,k_i}wB$ since it is equivalent to consider the image $Nv_i$ or the preimage $Nv_{\ell}=v_i$. Also, we observe that $\tau_{i,\ell}wB$ satisfies that $Nv_{\ell}=v_i$ with $\ell<i\leq h(\ell)$. Hence it is enough to consider the case $Nv_i=v_{k_i}$ with $i<k_i\leq h(i)$. Finally, $wB$ does not belong to $\hes(N,h)$ if and only if there is $i$ such that $Nv_i=k_i$ with $k_i>h(i)$. Then $\tau_{i,k_i}wB$ belongs to $\hes(N,h)$. Thus, all these fixed points are already included.  We summarize this information as follows:
\begin{enumerate}
    \item If $h=h_0$ and $Nv_i=v_{i-1}$, then $v_i=e_i$,
    \item If $h\neq h_0$, we have $Nv_i=v_{k_i}$ with $i+1\leq k_i \leq h(i)$ and
    \item for all $v_{\ell}$ such that $i\leq \ell \leq h(i)$, we have $Nv_{j}=v_{\ell}$ with $j_{\ell}\leq j\leq i-1$.
\end{enumerate}
The first condition is equivalent to $(x_1 -i)=0$. The remaining conditions are the next two determinants
\begin{align}\label{eq4.15}
\begin{aligned}
0&=d_{(1,\ldots, n),(v_1, v_2,\ldots,v_{h(i)},Nv_i , \ldots v_n)},
\end{aligned}
\end{align} 
and
\begin{align}
\begin{aligned}
0&=\prod_{\substack{Nv_j=v_{\ell}\\
i\leq\ell\leq h(i-1)}}\left( \prod_{j=j_{\ell}}^{i-1} {d_{(1,\ldots, n),(v_1, v_2,\ldots,v_{h(j)},Nv_j , \ldots v_n)}}\right).
\end{aligned}
\end{align} 
Hence, in general we have the determinant
\begin{align}
\begin{aligned}
\tilde{g}_i&=(x_i -i)\prod_{k=i+1}^{h(i)}(x_i-x_k -1)\cdot \prod_{\ell=i}^{h(i-1)}\left( \prod_{j=j_{\ell}}^{i-1} (x_j-x_{\ell}-1)\right).
\end{aligned}
\end{align} 
The polynomial $\tilde{g}_i$ allow us identify all possible images of $Nv_i$ and the preimages of $v_{\ell}$, $i\leq \ell\leq h(i)$.

Defining
\begin{equation}\label{eq4.21}
\tilde{I}_{h}=\langle \tilde{g_1},\ldots,\tilde{g_n}\rangle\subseteq \k[x_1,\ldots,x_n]
\end{equation}
it follows that $V(\tilde{I}_h)\cap \overline{S_n}$, where $\overline{S_n}=\lbrace \overline{w}=(w(1),\ldots, w(n)): w\in S_n\rbrace$ is the $S$-fixed point set of $\hes(N,h)$. By Theorem \ref{teo4}, we have 
\begin{equation}\label{eq5.13}
wB \in \hes(N,h)^S \Longleftrightarrow \overline{w}\in V(\tilde{I}_h)\cap \overline{S_n}. 
\end{equation}
This means, $\tilde{I}_h$ identifies the fixed points of $\hes(N,h)$. On the other hand, for the polynomials
\begin{equation}\label{eq4.26}
g_i=(x_i -it)\left(\prod_{k<i,j\leq i}(x_k -x_j-t)\right)\left(\prod_{j>i}(x_i-x_j -t)\right).
\end{equation} 
of $\k[\boldsymbol{x},t]=\k[x_1,\ldots,x_n,t]$, where $\k[t]$ is the Lie algebra of the torus $S$, 
the ideal of \cite[Algorithm 5.10]{phd}
$$I_h=\langle g_1,\ldots, g_n\rangle\subseteq \k[\boldsymbol{x},t]$$ 
satisfies that, if $x_i=w(i)t$ for $w\in S_n$, we have
\begin{equation}\label{eq4.24.1}
g_i=t\tilde{g}_i.
\end{equation}
Our goal is to give a different proof of \cite[Theorem 5.11]{phd} using the ideal $\tilde{I}_h$. By \cite[Chapter 5]{phd}
\begin{equation}\label{eq4.25}
H_{S}^*(\GL_n(\k)/B;\k)\simeq\k[{\boldsymbol{x}},t]/E_n({\boldsymbol{x}},t),
\end{equation} 
where $E_n({\boldsymbol{x}},t)={\langle e_i(x_1,\ldots,x_n)-e_i(t): 1\leq i\leq n\rangle}$ is the ideal generated by 
  the elementary symmetric functions $e_i(\boldsymbol{x})$ in the variables $x_1, \ldots, x_n$, and the elementary symmetric functions $e_i(t)=e_i(t^1,\ldots,t^n)$ in powers of $t$. 
The ideal $E_n({\bf{x}},t)$ detects the $S$-fixed points on $\GL_n(\k)/B$ by \cite[Theorem 3.1]{GM}. Moreover,
\begin{equation}\label{eq4.22}
\Spec(H^*(\GL_n(\k)/B;\k))\simeq \bigcup_{wB\in\GL_n(\k)/B} \fft_{w}\subset H_{2}^S(\GL_n(\k)/B;\k)\cong \k^n
\end{equation} 
with $\fft_w$  the line corresponding to the fixed point $wB \in \GL_n(\k)/B$. Since all zeros of $E_n(\boldsymbol{x},t)$ are of the form $wB$ with $w\in S_n$, all zeros of $E_n(\boldsymbol{x},t)+I_h$ are of the form $wB$ with $w\in S_n$. Then 
by \eqref{eq4.24.1}, $\overline{w}\in V(\tilde{I}_h)\cap \overline{S_n}$ if and only if $wB\in V(E_n(\boldsymbol{x},t)+I_h)$. Hence, by \eqref{eq5.13} $E_n(\boldsymbol{x},t)+I_h$ identifies $\hes(N,h)^S$. Since $\hes(N,h)$ is a $GM$-subspace of $\GL_n(\k)/B$ it follows that $\Spec(H_{S}^*(\hes(N,h);\k))$ is a subarrangement of $\Spec(H_{S}^*(\GL_n(\k)/B;\k)$. Thus, 
\begin{equation}\label{eq4.23}
\Spec(H_{S}^*(\hes(N,h);\k))\simeq \Spec\big(H_{S}^*(\GL_n(\k)/B;\k)/I\big)\simeq \bigcup_{wB\in \hes(N,h)}\!\!\!\!\fft_w
\end{equation} 
where $I=E_n(\boldsymbol{x},t)+I_h$ identifies the fixed points in $\hes(N,h)$ since
\begin{equation}
    wB\in \hes(N,h)^S \Leftrightarrow \overline{w}\in V(\tilde{I}_h \cap \overline{S_n}) \Leftrightarrow wB\in V(E_n(\boldsymbol{x},t)+I_h).
\end{equation} 
Thus, by  \eqref{eq4.23}, it follows
$$E_n({\bf{x}},t) + I_h=\bigcap_{wB\in \hes(N,h)} I(\fft_w).$$
We have proved:

\begin{thm}[{\cite[Theorem 5.11]{phd}}]\label{Insko2}
Let $I(\fft_w)$ denote the ideal of the line $\fft_w$ of \eqref{eq4.22}
for each $wB\in (\GL_n(\k)/B)^S$. Then,
$$E_n({\bf{x}},t) + I_h=\bigcap_{wB\in \hes(N,h)} I(\fft_w)$$
In the other words, the equivariant cohomology of the regular nilpotent Hessenberg variety $\hes(N,h)$ is
\begin{equation}\label{eq4.28}
\begin{aligned}
H_{S}^*(\hes(N,h);\k)\simeq \frac{\k[{\bf{x}},t]}{\bigcap_{wB\in \hes(N,h)}I(\fft_w)} \simeq \k[{\bf{x}},t]/(E_n({\bf{x}},t) + I_h).
\end{aligned}
\end{equation} 
\end{thm}

\end{document}